\documentclass[a4paper,11pt]{amsart}
\usepackage[pdftex]{color,graphicx}
\usepackage{amssymb}
\usepackage{amsfonts}
\usepackage{esint}
\usepackage{amsmath}
\usepackage{amsrefs}
\usepackage{epsfig}
\usepackage{amsthm}
\usepackage{color}
\usepackage[]{epsfig}
\usepackage[]{pstricks}
\usepackage[utf8]{inputenc}

\newtheorem{theorem}{Theorem}[section]

\newtheorem{lemma}{Lemma}[section]

\newtheorem{example}{Example}[section]
\newtheorem{corollary}{Corollary}[section]

\newtheorem{remark}{Remark}[section]

\newcommand{\s}{\mathbb{S}}
\newcommand{\R}{\mathbb{R}}

\newcommand{\Ric}{Ric}
\newcommand{\tr}{tr}
\newcommand{\Vol}{Vol}

\newcommand{\vol}{vol}
\newcommand{\ind}{ind}
\newcommand{\Ind}{Ind}

\numberwithin{equation}{section}
\title[Essential Spectrum of Weighted Laplacian]{Essential Spectrum of the Weighted Laplacian on Noncompact Manifolds and Applications}
\author{Adina Rocha}
\date{Received: May 04, 2016 / Accepted: July 21, 2016 }

\address{Instituto de Matem\'atica,
Universidade Federal de Alagoas,
Macei\'o, AL, 57072-900, Brazil}
\email{adina@pos.mat.ufal.br}

\begin{document}
\footnotetext{A. Rocha was partially supported by CAPES of Brazil}

\begin{abstract}
We obtain upper estimates for the bottom (that is, greatest lower bound) of the essential spectrum of weighted Laplacian operator of a noncompact weighted manifold under assumptions of the volume growth of their geodesic balls and spheres. Furthermore, we find examples where the equality occurs in the estimates obtained. As a consequence, we give estimates for the weighted mean curvature of complete noncompact hypersurfaces into weighted manifolds.\\

\end{abstract}
\maketitle

\section{Introduction}
\label{intro}

The $L^{2}$ spectrum of the Laplacian operator, denoted by $\sigma(-\Delta)$, has been analyzed and computed for large classes of complete noncompact manifolds. For example, when that manifold has a soul whose exponential map is a diffeomorphism, supposing nonnegative sectional
curvature and some other additional conditions under dimension of manifold, Escobar \cite{Escobar1986} and Escobar and Freire \cite{EscobarFreire1992} proved that
$$\sigma(-\Delta)=\sigma_{ess}(-\Delta)=[0,+\infty),$$  where  $\sigma_{ess}(-\Delta)$ denotes the essential spectrum of Laplacian operator. The additional conditions under dimension of $M$
assumed by them were proved to be unnecessary by Zhou \cite{Zhou1994}. Now, by assuming that the complete noncompact manifold has nonnegative Ricci curvature and Euclidean volume growth, the essential spectrum of the Laplacian operator was proved to be $[0,+\infty)$ by Donnelly \cite{Donnelly1997}. It is classical that the essential spectrum of Laplacian operator on compact manifold is empty, and thus, its $L^2$ spectrum is discrete. 

Given that the $L^2$ spectrum of Laplacian operator for a large class of noncompact manifolds is essential, it becomes important to estimate the bottom (that is, the greatest lower bound) of essential spectrum of $-\Delta$ denoted by $\inf\sigma_{ess}(-\Delta)$. There exists many interesting results on the estimates for $\inf\sigma_{ess}(-\Delta)$, for instance, \cite{Cheng1975}, \cite{Pinsky1978}, \cite{Pinsky1979}, 
\cite{DonnellyLi1979}, 
\cite{Brooks1981}, \cite{Brooks1984}, \cite{Higuchi2001}, and \cite{LiWang2002}.

The problem of estimating the $\inf\sigma_{ess}(-\Delta)$ of complete noncompact manifold $M$ under simple geometric assumptions has been intensively studied in the past decades. For example, Donnelly proved in \cite{Donnelly1981} that $$\inf\sigma_{ess}(-\Delta)\leq(n-1)^2k/4$$ when the Ricci curvature is bounded from below by the constant $-(n-1)k,$ where $n=\dim M$ and $k\geq0.$ In \cite{Brooks1981} and \cite{Brooks1984}, Brooks generalized this Donnelly's estimate. He showed that if $M$ has infinite volume, then $$
\inf\sigma_{ess}(-\Delta)\leq \bar\mu_v^2/4,
$$ where \ $\bar\mu_v=\limsup_{r\rightarrow\infty}\frac{1}{r}\log\Vol(B_r)$; and if $M$ has finite volume, $$
\inf\sigma_{ess}(-\Delta)\leq \bar\mu_w^2/4,
$$ where $\bar\mu_w=\limsup_{r\rightarrow\infty}\frac{-1}{r}\log(\Vol(M)-\Vol(B_r))$.
Later, Higuchi \cite{Higuchi2001} improved these Brooks' estimates. 

In many occasions, it is natural to consider a $n-$dimensional Riemannian manifold $(M^n,\langle\,,\,\rangle)$ endowed with a weighted measure of the form $e^{-f}d\sigma$, where $f$ is a smooth function on $M$, called {\it weight function}, and $d\sigma$ is the volume element induced by the metric $\langle\,,\,\rangle$.
A {\it weighted manifold} is a triple $${M}_f^{n}=( M^{n},\langle\,,\,\rangle,e^{-f}d\sigma).$$ The {\it weighted Laplacian operator} $\Delta_f$, defined by
$$\Delta_fu:=\Delta u-\langle\nabla f,\nabla u\rangle,$$
is associated to $e^{-f}d\sigma$ as well as $\Delta$ is associated to $d\sigma.$ Moreover, $\Delta_f$ is a self-adjoint operator on the  $L^{2}_f$ space of square integrable functions on $M$ with respect to the measure $e^{-f}d\sigma,$ and therefore, the $L^{2}_f$ spectrum of $\Delta_f$ on $M$, denoted by $\sigma(-\Delta_f)$, is a subset of $[0,+\infty).$ Also, $\sigma(-\Delta_f)$ can be decomposed into the disjoint union $\sigma_d(-\Delta_f)\cup\sigma_{ess}(-\Delta_f)$, where $\sigma_d(-\Delta_f)$ is the set of isolated eigenvalues of finite multiplicity, called {\it discrete spectrum}, and its complement $\sigma_{ess}(-\Delta_f),$ called {\it essential spectrum}, is the set of eigenvalues of infinite multiplicity and accumulation points of the spectrum.

A natural extension
of the Ricci curvature tensor to this new context is the Bakry-\'{E}mery Ricci curvature
tensor, see \cite{BakryEmery1985}, given by
$$\Ric_f=\Ric+\nabla^{2} f,$$
where $\nabla^{2} f$ is the hessian of $f$ on $M$. It is known that a complete weighted manifold satisfying $\Ric_f\geq\displaystyle{\lambda}$ for some constant $\lambda>0$ is not necessarily compact. One of the examples is the Gaussian shrinking soliton $\left(\R^n,g_{can},e^{-\frac{|x|^2}{4}}d\sigma\right)$ with the canonical metric $g_{can}$ and $\Ric_f=\displaystyle{\frac{1}{2}g_{can}}.$ If $M_f$ is complete noncompact manifold with $\Ric_f\geq0$ and  $\lim_{r\rightarrow\infty}\frac{|f|}{r}=0$, the $L^{2}_f$ essential spectrum  of $-\Delta_f$ was proved to be $[0,+\infty)$ by Silvares \cite{Silvares2014}. Moreover, for a  compact weighted manifold with $\Ric_f\geq \lambda$ for constant $\lambda>0$, the $L^{2}_f$ spectrum  of $-\Delta_f$ is discrete, see \cite{HeinNaber2014}. 

Let $B_r$ be the geodesic ball of $M$ with center in a fixed point $o\in M$  and radius $r>0$. The {\it weighted volume} of $B_r$ is given by $$\Vol_f(B_r)=\int_{B_r}e^{-f}d\sigma$$ 
and the {\it weighted volume} of $M$ is defined by $$\Vol_f(M)=\int_{M}e^{-f}d\sigma.$$ Manifolds with the Bakry-\'{E}mery Ricci curvature bounded below have been studied
by many authors in recent years, particularly some interesting weighted volume estimates and splitting theorems, which can be
found in \cite{WeiWylie2009} and \cite{MunteanuWang2014}, for example. Some results about estimates of weighted volume can be also seen in \cite{ChengMejiaZhou2015} and \cite{ChengZhou2013}. 

The aim of this paper is to give estimates for $\inf\sigma_{ess}(-\Delta_f)$ for complete noncompact manifolds in function of the weighted volume growth of their geodesic balls and spheres. Precisely, we have the following theorem.

\begin{theorem}\label{it4} Let $M_f$ be a complete noncompact weighted manifold. 
	\begin{enumerate}
		\item[(i)] If $\Vol_f(M)=+\infty$, then
		$$
		\inf\sigma_{ess}(-\Delta_f)\leq \frac{\mu_v^2}{4},
		$$
		where \ $		\mu_v=\displaystyle{\liminf_{r\rightarrow\infty}\frac{1}{r}\log\Vol_f(B_r)}$.\vspace{0.2cm}	
		\item[(ii)] If $\Vol_f(M)<+\infty$, then
		\begin{equation}\label{al7}
		\inf\sigma_{ess}(-\Delta_f)\leq \frac{\mu_w^2}{4},
		\end{equation}
		where \  $\mu_w=\displaystyle{\liminf_{r\rightarrow\infty}\frac{-1}{r}\log(\Vol_f(M)-\Vol_f(B_r))}.$
	\end{enumerate}	
	Furthermore, if $M_f=(\R^{n},ds^2,e^{-f}d\sigma)$ with $$f=\frac{r}{2} \ \ \ \ {\rm and} \ \ \ \ ds^2=dr^2+g^2(r)d\theta^2$$
	such that $g:[0,+\infty)\rightarrow\R$ is a nonnegative smooth function satisfying
	\begin{equation}\label{eee2}
	g(0)=0, \ \ \ g'(0)=1, \ \ \ \normalfont{and} \ \ \ g(r)=e^{-\frac{r}{2(n-1)}} \ \ \normalfont{for \ all} \ r\geq r_0>0,
	\end{equation}
	then occurs equality in {\normalfont{(\ref{al7})}}. Here $d\theta^2$ denotes the standard metric on $(n-1)$-dimensional unit sphere $\s_1^{n-1}$ and $r$ is the Euclidian distance to origen.	 
\end{theorem}

\begin{remark}
	In Theorem 6.2 of \cite{BessaPigolaSetti2013}, Bessa, Pigola, and Setti gave an upper bounds for $\inf\sigma_{ess}(-\Delta_f)$ in terms of the volume growth. Indeed, they show that:
	\begin{enumerate}
		\item[(i)] If $\Vol_f(M)=+\infty$, then
		$$
		\inf\sigma_{ess}(-\Delta_f)\leq \displaystyle{\limsup_{r\rightarrow\infty}\frac{1}{r}\log\Vol_f(B_r)}.
		$$
		\item[(ii)] If $\Vol_f(M)<+\infty$, then
		$$
		\inf\sigma_{ess}(-\Delta_f)\leq \displaystyle{\limsup_{r\rightarrow\infty}\frac{-1}{r}\log(\Vol_f(M)-\Vol_f(B_r))}.
		$$	
	\end{enumerate}
	Therefore, Theorem \ref{it4} gives an improved estimate for $\inf\sigma_{ess}(-\Delta_f)$. 
\end{remark}

Now, by supposing $f=C$ been a constant function,  we have 
$$\Delta_f=\Delta, \ \ \ \ \Vol_f=e^{-C}\Vol, \ \ \ \ {\rm and} \ \ \ \ \vol_f=e^{-C}\vol.$$
Moreover, $\inf\sigma_{ess}(-\Delta_f)=\inf\sigma_{ess}(-\Delta)$ is the bottom of essential spectrum of the Laplacian acting on $L^2(M)$. Therefore, as a consequence of Theorem \ref{it4}, it follows: 

\begin{corollary}\label{itc11} Let $M$ be a noncompact complete manifold.  
	\begin{enumerate}
		\item[(i)] If $\Vol(M)=+\infty$, then
		$$
		\inf\sigma_{ess}(-\Delta)\leq \frac{\mu_v^2}{4},
		$$
		where \ $\mu_v=\displaystyle{\liminf_{r\rightarrow\infty}\frac{1}{r}\log\Vol(B_r)}$.\vspace{0.2cm}
		\item[(ii)] If $\Vol(M)<+\infty$, then
		$$
		\inf\sigma_{ess}(-\Delta)\leq \frac{\mu_w^2}{4},
		$$
		where \ $\mu_w=\displaystyle{\liminf_{r\rightarrow\infty}\frac{-1}{r}\log(\Vol(M)-\Vol(B_r))}.$	
	\end{enumerate}		
\end{corollary}

\begin{remark}
	Corollary \ref{itc11} was proved by Higuchi, see Corollary 2 of \cite{Higuchi2001}.
\end{remark}

Let $\Omega\subset M$ be a compact domain. The bottom of the spectrum of $\Delta_f$ on $M\setminus\Omega$ with the Dirichlet boundary condition on $\partial\Omega$ admits the usual variational characterization 
\begin{eqnarray*}
	\lambda_1^f(M\setminus\Omega)=\inf_{u\in C_c^{\infty}(M\setminus\Omega)}\frac{\displaystyle{\int_{M\setminus\Omega}|\nabla u|^2}e^{-f}d\sigma}{\displaystyle{\int_{M\setminus\Omega}u^2}e^{-f}d\sigma},
\end{eqnarray*}
where $ C_c^{\infty}(M\setminus\Omega)$ denotes the set of the smooth functions $u:M\setminus\Omega\rightarrow\R$ with compact support on $M\setminus\Omega$. Moreover, it can be seen in Theorem 6.1 of \cite{BessaPigolaSetti2013} that
\begin{equation}\label{1111}
\inf\sigma_{ess}(-\Delta_f)=\sup_{\Omega}\lambda_1^f(M\setminus\Omega),
\end{equation}
where $\Omega$ runs over the set of compact domains of $M$.

Let $\partial B_r$ be the geodesic sphere of $M$ with center in a fixed point $o\in M$  and radius $r>0$. The {\it weighted volume} of $\partial B_r$ is given by $$\vol_f(\partial B_r)=\int_{\partial B_r}e^{-f}dA,$$
where $dA$ is the volume form on $\partial B_r$.  

Now, we are ready to enunciate the second result of this paper, namely:

\begin{theorem}\label{it1}
	Let $M_f$ be a complete noncompact  weighted manifold and let $\Omega$ be a compact subset of $M$. If there exists a positive constant real $\alpha$ such that 
	$$\left|\frac{d}{dr}\left(\log\vol_f(\partial B_r)\right)\right|\leq \alpha \ \ \ {\normalfont for \ all} \ r\geq r_0,$$
	then
	\begin{equation}\label{e10}
	\lambda_1^f(M\setminus\Omega)\leq\frac{\alpha^2}{4}.
	\end{equation} 
	Consequently,
	\begin{equation}\label{at11}
	\inf\sigma_{ess}(-\Delta_f)\leq\frac{\alpha^2}{4}.
	\end{equation} 
	Furthermore, if $M_f=(\R^{n},ds^2,e^{-f}d\sigma$) with $$f=\frac{\alpha r}{2} \ \ \ \ {\normalfont and} \ \ \ \ ds^2=dr^2+g^2(r)d\theta^2$$
	such that $g:[0,+\infty)\rightarrow\R$ is a nonnegative smooth function satisfying
	\begin{equation}\label{ei13}
	g(0)=0, \ \ \ g'(0)=1, \ \ \ {\normalfont and} \ \ \ g(r)=e^{-\frac{\alpha}{2(n-1)} r} \ \ {\normalfont for \ all} \ r\geq r_0>0,
	\end{equation}
	then occurs equality in {\normalfont (\ref{e10})} and {\normalfont (\ref{at11})} for any compact $\Omega\supset B_{r_0}.$ Here $d\theta^2$ denotes the standard metric on $(n-1)$-dimensional unit sphere $\s_1^{n-1}$ and $r$ is the Euclidian distance to origen.
\end{theorem}

\begin{remark} \label{ri1} The condition 
	$$\frac{d}{dr}\left(\log\vol_f(\partial B_r)\right)\leq \alpha, \ \ \  \ r>0,
	$$
	implies that 
	$$Vol_f(B_r)\leq Ce^{\alpha r},$$ where  $C=1/\alpha$. But $M$ does not necessarily have $Vol_f(M)=+\infty.$ Look to $M_f=(\R^{n},ds^2,e^{-f}d\sigma$) with $f=\alpha r/2$ and $ds^2=dr^2+g^2(r)d\theta^2$ such that $g:[0,+\infty)\rightarrow\R$ is a nonnegative smooth function satisfying \normalfont{(\ref{ei13})}; it has finite weighted volume, i.e.,
	\begin{eqnarray*}
		\Vol_f(B_r)=\Vol_f(B_{r_0})+\omega_n\int_{r_0}^re^{-\alpha t}dt=\Vol_f(B_{r_0})+\frac{\omega_n}{\alpha}\left(-e^{-\alpha r}+e^{-\alpha r_0}\right),
	\end{eqnarray*}    	
	and hence,
	$$\Vol_f(M)=\lim_{r\rightarrow\infty}\Vol_f(B_r)=\Vol_f(B_{r_0})+\frac{\omega_n}{\alpha e^{\alpha r_0}}.$$ 
	This can be viewed in expression (\ref{er1}) of Example \ref{ri3}.  
\end{remark}	

\begin{remark} There is a class of complete Riemannian manifolds that satisfies the conditions of Theorem \ref{it1} (see Example \ref{ri3}).
\end{remark}

Now, by supposing $f$ been a constant function, it follows directly from Theorem \ref{it1} that

\begin{corollary} Let $M$ be a noncompact complete manifold and let $\Omega$ be a compact subset of $M$. If 
	$$\left|\frac{d}{dr}\left(\log\vol(\partial B_r)\right)\right|\leq \alpha \ \ \ {\normalfont for} \ r\geq r_0,$$
	then
	\begin{equation}
	\label{aaa5}
	\lambda_1(M\setminus\Omega)\leq\frac{\alpha^2}{4}.
	\end{equation}
	Consequently,
	\begin{equation}
	\label{aaa6}
	\inf\sigma_{ess}(-\Delta)\leq\frac{\alpha^2}{4}.
	\end{equation}
	Furthermore, if $M=(\R^{n},ds^2)$ with $ds^2=dr^2+g^2(r)d\theta^2$ such that $g:[0,+\infty)\rightarrow\R$ is a nonnegative smooth function satisfying
	$$
	g(0)=0, \ \ \ g'(0)=1, \ \ \ {\normalfont and} \ \ \ g(r)=e^{-\frac{\alpha}{(n-1)} r} \ \ {\normalfont for \ all} \ r\geq r_0>0,
	$$
	then occurs equality in {\normalfont (\ref{aaa5})} and {\normalfont (\ref{aaa6})}
	for any compact $\Omega\supset B_{r_0}.$ Here $d\theta^2$ denotes the standard metric on $(n-1)$-dimensional unit sphere $\s_1^{n-1}$ and $r$ is the Euclidian distance to origen.
\end{corollary}

Before beginning to enunciate some applications, we need to introduce some definitions that will be necessary to understand the results.

Let ${\overline M}_f^{n+1}$ be a weighted manifold, i.e., $${\overline M}_f^{n}=( \overline M^{n},\langle\,,\,\rangle,e^{-f}d\mu).$$ Let $x:M^n\rightarrow\overline M^{n+1}_f$ be an isometric immersion of a Riemannian orientable manifold $M^n$ into weighted manifold \ $\overline M^{n+1}_f.$ \ The function \ $f:\overline M \rightarrow\R$ \ restricted to $M$ induces a weighted measure $e^{-f}d\sigma$ on $M$. Thus, we have an induced weighted manifold $M_f^n=(M,\langle\,,\,\rangle,e^{-f}d\sigma)$.

The {\it second fundamental form} $A$ of $x$ is defined by
$$A(X,Y)=(\overline\nabla_XY)^{\perp}, \ \ \ \ \ \ X,Y\in T_pM, \ \ p\in M,$$
where $\perp$ symbolizes the projection above the normal bundle of $M$. The {\it weighted mean curvature vector} of $M$ is defined by
$${\bf H}_f={\bf H}+(\overline\nabla f)^{\perp},$$
with ${\bf H}=\tr A.$ The hypersurface $M$ is called {\it $f$-minimal} when its weighted mean curvature vector ${\bf H}_f$ vanishes identically; and when there exists real constant $C$ such that ${\bf H}_f=-C\eta$ with $\eta$ being unit normal vector field, we say the hypersurface $M$ has {\it constant weighted mean curvature}. 

The operator 
$$L_f=\Delta_f+|A|^2+\overline{\Ric_f}(\eta,\eta)$$
is called the {\it $f$-stability operator} of the immersion $x$ and it is associated with the quadratic form
$$I_f(u,u)=-\int_MuL_fue^{-f}d\sigma.$$
For each compact domain $\Omega\subset M$, define the index, $\ind_f\Omega$, of $L_f$ in $\Omega$ as the maximal dimension of a subspace of $C^{\infty}_c(\Omega)$ where $I_f$ is a negative definite. The {\it index}, $\ind_fM$, of $L_f$ in $M$ (or simply, the index of $M$) is then defined by
$$\ind_fM=\sup_{\Omega\subset M}\ind_f\Omega,$$
where \ the \ supreme \ is \ taken \ over \ all \ compact \ domains \ $\Omega\subset M.$  For more details, see \cite{ChengZhou2015}. 

The following results are applications of the estimates of bottom essential spectrum of weighted Laplacian operator given in Theorem \ref{it4} and Theorem \ref{it1}.

\begin{theorem}\label{it5}
	Let $x:M^n\rightarrow\overline M_f^{n+1}$ be an isometric immersion of a complete noncompact manifold $M^n$ into  an oriented complete weighted manifold $\overline{M}_f^{n+1}$ with unit normal vector field  $\eta$. 
	Let
	$$		\mu_v=\displaystyle{\liminf_{r\rightarrow\infty}\frac{1}{r}\log\Vol_f(B_r)} \ \ \ {\normalfont and} \ \ \ \mu_w=\displaystyle{\liminf_{r\rightarrow\infty}\frac{-1}{r}\log(\Vol_f(M)-\Vol_f(B_r))}.
	$$
	If $M$ has constant weighted mean curvature $H_f$ and $\ind_fM<\infty$, then  
	\begin{equation}\label{aa12}
	\frac{H^2_f}{nm}\geq-\frac{\mu^2}{4}+\inf_{M\setminus B_{r}}\left\{\overline{\Ric}_f(\eta,\eta)+\frac{\langle\overline\nabla f,\eta\rangle^2}{n(1+m)}\right\}
	\end{equation}
	and
	$$\frac{H^2_f}{n(1+m)}\leq\frac{\mu^2}{4}-\inf_{M\setminus B_{r}}\left\{\overline{\Ric}_f^{nm}(\eta,\eta)\right\}$$
	for any constant $m>0$, where \
	$\mu=\mu_v$ if $\Vol_f(M)=+\infty$ and  $\mu=\mu_w$ if 	$\Vol_f(M)<+\infty$. Particularly, if $$\overline{\Ric}_f^{nm}(\eta,\eta)\geq\frac{\mu^2}{4}$$
	for some constant positive $m$, then $M$ is a $f$-minimal hypersurface.
\end{theorem}
Here
$$
{\overline{\Ric}_f}^{nm}=\overline{\Ric_f}-\frac{df\otimes df}{nm}, \ \ \ \ \ m>0,
$$
denotes a generalization of the Bakry-\' Emery Ricci curvature.

As a consequence of the inequality (\ref{aa12}) of Theorem \ref{it5}, it follows 

\begin{corollary}\label{aic1} 
	\ Let \ $\overline{M}_f^{n+1}$ be an oriented complete weighted manifold with $\overline{\Ric}_f\geq k>0$, where $k$ is a fixed constant. Then, there is no complete noncompact $f$-minimal hypersurface $M^{n}$ immersed into  $\overline{M}_f^{n+1}$  with $ind_fM<+\infty$ and satisfying either \
	$\mu_v<2\sqrt k$ \ if \ $\Vol_f(M)=+\infty$ \ or \  $\mu_w<2\sqrt k$ \ if \  $\Vol_f(M)<+\infty$.   
\end{corollary}

It is said that the weighted volume of $M$ has {\it polynomial growth} if there exists positive numbers $\alpha$, $C$, and $R_0$  such that $$\Vol_f(B_r)\leq Cr^{\alpha}$$ for any $r\geq R_0.$  Thus,
$$0\leq\mu_v=\displaystyle{\liminf_{r\rightarrow\infty}\frac{1}{r}\log\Vol_f(B_r)}\leq\displaystyle{\liminf_{r\rightarrow\infty}\frac{1}{r}\log(Cr^{\alpha})}=0,$$
i.e., $\mu_v=0$. Therefore, it follows from Theorem \ref{it5} the following consequence:	

\begin{corollary}\label{ic6}
	Let $x:M^n\rightarrow\overline M_f^{n+1}$ be an isometric immersion of a complete noncompact manifold $M^n$ into  an oriented complete weighted manifold $\overline{M}_f^{n+1}$ with unit normal vector field  $\eta$. Assume that the weighted volume of $M$ is infinite and it has polynomial growth. If $M$ has constant weighted mean curvature $H_f$ and $\ind_fM<\infty$, then there exists a constant $r_0>0$ such that for all $r\geq r_0$,
	\begin{equation}\label{e13}
	\frac{H^2_f}{nm}\geq\inf_{M\setminus B_{r}}\left\{\overline{\Ric}_f(\eta,\eta)+\frac{\langle\overline\nabla f,\eta\rangle^2}{n(1+m)}\right\}
	\end{equation}
	and
	$$\frac{H^2_f}{n(1+m)}\leq-\inf_{M\setminus B_{r}}\left\{\overline{\Ric}_f^{nm}(\eta,\eta)\right\},$$
	where $m$ is a positive constant. In particular, if $$\overline{\Ric}_f^{nm}(\eta,\eta)\geq0$$
	for some real $m>0,$ then $M$ is a $f$-minimal hypersurface.
\end{corollary}

Observe also that if we assume $Vol_f(B_r)\leq Ce^{\alpha r}$, then
$$\mu_v=\displaystyle{\liminf_{r\rightarrow\infty}\frac{1}{r}\log\Vol_f(B_r)}\leq\displaystyle{\liminf_{r\rightarrow\infty}\frac{1}{r}\log(Ce^{\alpha r})}=\alpha.$$
Therefore, as a consequence of Theorem \ref{it5}, we obtain the following:

\begin{corollary}\label{ic7}
	Let $x:M^n\rightarrow\overline M_f^{n+1}$ be an isometric immersion of a complete noncompact manifold $M^n$ into  an oriented complete weighted manifold $\overline{M}_f^{n+1}$ with unit normal vector field  $\eta$. Assume that the weighted volume of $M$ is infinite and it satisfies $Vol_f(B_r)\leq Ce^{\alpha r}$ for some constant $C$, $\alpha$, and $r\geq 0$. If $M$ has constant weighted mean curvature $H_f$ and $\ind_fM<\infty$,  then  there exists a constant $r_0>0$ such that for all $r\geq r_0$,
	\begin{equation}\label{e122}
	\frac{H^2_f}{nm}\geq-\frac{\alpha^2}{4}+\inf_{M\setminus B_{r}}\left\{\overline{\Ric}_f(\eta,\eta)+\frac{\langle\overline\nabla f,\eta\rangle^2}{n(1+m)}\right\}
	\end{equation}
	and
	$$\frac{H^2_f}{n(1+m)}\leq\frac{\alpha^2}{4}-\inf_{M\setminus B_{r}}\left\{\overline{\Ric}_f^{nm}(\eta,\eta)\right\},$$
	where $m$ is a positive constant. Particularly, if $$\overline{\Ric}_f^{nm}(\eta,\eta)\geq\frac{\alpha^2}{4}$$ for some real $m>0,$ then $M$ is a $f$-minimal hypersurface.
\end{corollary}

\begin{remark}
	When $f$ is a constant function, Corollary \ref{ic6} was obtained by Alencar and do Carmo, see Theorem 1.1 of \cite{AlencardoCarmo1993}, and improved by do Carmo and Zhou, see Theorem 4.1 of \cite{doCarmoZhou1999}. Already, Corollary \ref{ic7} was proved by do Carmo and Zhou, see Theorem 4.4 of \cite{doCarmoZhou1999}.
\end{remark}

By using the inequality (\ref{e13}) of the Corollary \ref{ic6} and the inequality (\ref{e122}) of the Corollary \ref{ic7}, we can acquire the next two results:

\begin{corollary}\label{ic1} Let $\overline{M}_f^{n+1}$ be an oriented complete weighted manifold with $\overline{\Ric}_f\geq k>0$, where $k$ is a fixed constant. Then, there is no complete noncompact $f$-minimal hypersurface $M^{n}$ immersed into $\overline{M}_f^{n+1}$ with $Vol_f(M)=+\infty$, $ind_fM<+\infty$, and polynomial growth weighted volume.
\end{corollary}

\begin{corollary}\label{ic8} Let $\overline{M}_f^{n+1}$ be an oriented complete weighted manifold with $\overline{\Ric}_f\geq k>0$, where $k$ is a fixed constant. Then there is no complete noncompact $f$-minimal hypersurface $M^{n}$ immersed into $\overline{M}_f^{n+1}$ with $ind_fM<+\infty$, $\Vol_f(M)=+\infty$,  and  $Vol_f(B_r)\leq Ce^{\alpha r}$ for any $r\geq0$ and $\alpha<2\sqrt k$.
\end{corollary}

Now, using the estimate of $\lambda_1(M\setminus\Omega)$ viewed in Theorem \ref{it1}, we obtain the application:
\begin{theorem}\label{it3}
	Let $x:M^n\rightarrow\overline M_f^{n+1}$ be an isometric immersion of a complete noncompact manifold $M^n$ into  an oriented complete weighted manifold $\overline{M}_f^{n+1}$ with unit normal vector field  $\eta$. Assume that $$\left|\displaystyle{\frac{d}{dr}}\left(\log\vol_f(\partial B_r)\right)\right|\leq\alpha \ \ {\normalfont for \ all} \ r\geq t_0>0.$$
	If $M$ has constant weighted mean curvature $H_f$ and $\ind_fM<\infty$,  then there exists a constant $r_0>0$ such that for all $r\geq r_0$,
	\begin{equation}\label{e12}
	\frac{H^2_f}{nm}\geq-\frac{\alpha^2}{4}+\inf_{M\setminus B_{r}}\left\{\overline{\Ric}_f(\eta,\eta)+\frac{\langle\overline\nabla f,\eta\rangle^2}{n(1+m)}\right\}
	\end{equation}
	and
	$$\frac{H^2_f}{n(1+m)}\leq\frac{\alpha^2}{4}-\inf_{M\setminus B_{r}}\left\{\overline{\Ric}_f^{nm}(\eta,\eta)\right\}.$$
	Particularly, if $$\overline{\Ric}_f^{nm}(\eta,\eta)\geq\frac{\alpha^2}{4}$$ for some real $m>0,$ then $M$ is a $f$-minimal hypersurface.
\end{theorem}

As a consequence of the inequality (\ref{e12}) of Theorem \ref{it3}, it follows 

\begin{corollary}\label{ic2} Let $\overline{M}_f^{n+1}$ be an oriented complete weighted manifold with $\overline{\Ric}_f\geq k>0$, where $k$ is a fixed constant. Then, there is no complete noncompact $f$-minimal hypersurface $M^{n}$ immersed into $\overline{M}_f^{n+1}$ with $ind_fM<+\infty$ and $$\left|\displaystyle{\frac{d}{dr}}\left(\log\vol_f(\partial B_r)\right)\right|\leq \alpha <2\sqrt k$$
	for all $r\geq t_0>0.$  
\end{corollary}

\section{Estimates for the Bottom of Essential Spectrum}

Let $M_f=(M^{n},\langle\,,\,\rangle,e^{-f}d\sigma)$ be a weighted manifold and let $K\subset M$ be a compact set. For a positive number $\delta>0,$ define the set 
$$
A_{\delta}(\partial K)=\{x\in M\setminus K; \ \ \rho(x,\partial K)\leq\delta\}
$$
with $\rho(x,\partial K)$ denoting the distance between $x$ and $\partial K.$ Moreover, consider the amount
\begin{equation}\label{a12}
\mu_{\delta}(r)=\frac{1}{r}\log \Vol_f(A_{\delta}(\partial B_r)),
\end{equation}
where $\partial B_r$ is the geodesic sphere of radius $r>0$ and center in a fixed point $o\in M$.

Established the above notations, we get the following:

\begin{lemma}\label{al1}
	Let $M_f$ be a complete noncompact weighted manifold. For any fixed $\delta>0$, we set
	\begin{equation}\label{a13}
	\mu_{\delta}=\liminf_{r\rightarrow\infty}\mu_{\delta}(r) \ \ \ \ \ \ {\normalfont and} \ \ \ \ \ \ \ 
	\bar\mu_{\delta}=\limsup_{r\rightarrow\infty}\mu_{\delta}(r),	
	\end{equation}		
	with $\mu_{\delta}(r)$ defined in \normalfont{(\ref{a12})}. \vspace{0.3cm}
	\begin{enumerate}
		\item[(i)] 	If $\Vol_f(M)=+\infty$, then $\inf\sigma_{ess}(-\Delta_f)\leq\mu_{\delta}^2/4.$ Besides that, \\  $\inf\sigma_{ess}(-\Delta_f)=0$ if $\mu_{\delta}<0.$\vspace{0.3cm}
		\item[(ii)] If $\Vol_f(M)<+\infty$, then $\inf\sigma_{ess}(-\Delta_f)\leq \bar\mu^2_{\delta}/4.$ 	
	\end{enumerate}
\end{lemma}

\begin{proof}
	For arbitrary compact domain $\Omega\subset M$, let $\lambda_1^f(M\setminus\Omega)$ be the bottom of the spectrum of $\Delta_f$ on $M\setminus\Omega$. It is well known that
	\begin{eqnarray*}
		\lambda_1^f(M\setminus\Omega)=\inf_{u\in C^{\infty}_c(M\setminus\Omega)}\frac{\displaystyle{\int_M}|\nabla u|^2e^{-f}d\sigma}{\displaystyle{\int_M}u^2e^{-f}d\sigma} 
	\end{eqnarray*}
	and, see \cite{BessaPigolaSetti2013},
	\begin{eqnarray*}
		\inf\sigma_{ess}(-\Delta_f)=\sup_{\Omega\subset M}\lambda_1^f(M\setminus\Omega).
	\end{eqnarray*}	
	Therefore it suffices to prove the following: for any fixed $\delta>0$ and compact domain $\Omega\subset M$, and for arbitrary sufficiently small $\varepsilon,\varepsilon_1>0,$ there exists a function $u$ with compact support in $M\setminus \Omega$ such that 
	\begin{eqnarray*}
		\frac{\displaystyle{\int_M}|\nabla u|^2e^{-f}d\sigma}{\displaystyle{\int_M}u^2e^{-f}d\sigma}<\alpha^2(\varepsilon)+\varepsilon_1, 
	\end{eqnarray*}
	where $\alpha^2(\varepsilon)\rightarrow\mu_{\delta}^2/4$ as $\varepsilon\rightarrow0$.	
	
	Let's establish a test function $u(x)=e^{h_j(x)}\cdot\chi_r(x)$ with compact support in $M\setminus \Omega$. To this, let $o\in M$ be a fixed point and let $\rho(x)=\rho(x,o)$ denote the distance from $o$ to $x\in M$. We define $\chi_r$, see Figure \ref{fig1}, and $h_j$, see Figure \ref{fig2}, as follows: for $r$ sufficiently large such that $\Omega\subset B_{r-\delta}$,	
	\begin{eqnarray*}
		\chi_r(x)&=&\left\{
		\begin{array}{ll}
			0,& {\rm if} \ x\in \Omega \ {\rm or} \ \rho(x)>r+\delta,\\
			\rho(x,\Omega)/\delta, & {\rm if} \ 0<\rho(x,\Omega)\leq\delta,\\
			1-\rho(x,B_r)/\delta, \ \ \ \ \ \ & {\rm if} \ r<\rho(x)\leq r+\delta,\\
			1, &  {\rm otherwise},
		\end{array}  
		\right.	
	\end{eqnarray*}
	
	\begin{figure}
		\resizebox{0.8\hsize}{!}{\includegraphics*{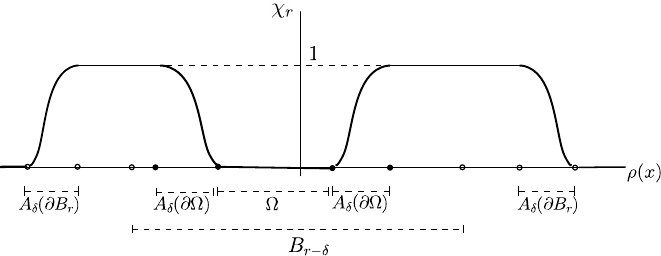}}
		\caption{The function $\chi_r$ with compact support in $M\setminus \Omega$}\label{fig1}
	\end{figure}
	
	\noindent
	and for a fixed number $\alpha\geq0$, and for a positive integer $j$,
	\begin{eqnarray*}
		h_j(x)&=&\left\{
		\begin{array}{ll}
			\alpha\rho(x),& {\rm if} \ \rho(x)\leq j,\\
			2\alpha j-\alpha\rho(x), \ \ \ & {\rm if} \ \rho(x)> j.
		\end{array}  
		\right.	
	\end{eqnarray*}	
	\begin{figure}
		\resizebox{0.6\hsize}{!}{\includegraphics*{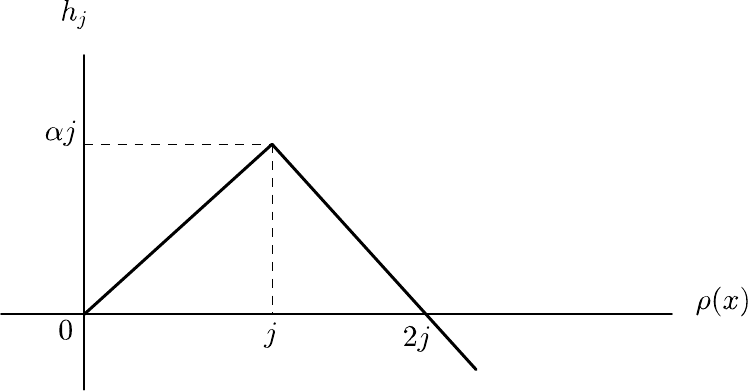}}
		\caption{The function $h_j$}\label{fig2}
	\end{figure}
	For $u=e^{h_j}\chi_r$, we have that
	$$\nabla u=e^{h_j}\nabla h_j\cdot\chi_r+e^{h_j}\nabla\chi_r$$
	and
	\begin{eqnarray*}\label{ee9}
		\int_{M\setminus\Omega}|\nabla u|^2e^{-f}d\sigma&=&\int_{M\setminus \Omega}e^{2h_j}(|\nabla h_j\cdot\chi_r+\nabla\chi_r|^2)e^{-f}d\sigma\nonumber\\
		&=&\int_{M\setminus \Omega}u^2|\nabla h_j|^2e^{-f}d\sigma+\int_{M\setminus \Omega}e^{2h_j}(2\chi_r\cdot\langle\nabla h_j,\nabla\chi_r\rangle+|\nabla \chi_r|^2)e^{-f}d\sigma.\nonumber\\
		&\leq&\int_{M\setminus \Omega}u^2|\nabla h_j|^2e^{-f}d\sigma+\int_{M\setminus \Omega}e^{2h_j}(2\chi_r\,|\nabla h_j|\,|\nabla\chi_r|+|\nabla \chi_r|^2)e^{-f}d\sigma.\nonumber\\
	\end{eqnarray*}
	Note that $\nabla\chi_r$ is supported in $A_{\delta}(\partial B_r)\cup A_{\delta}(\partial \Omega)$ and $|\nabla\chi_r|\leq1/\delta$. In addition, $|\nabla h_j|\leq\alpha$. Therefore, 
	\begin{eqnarray}
	\lefteqn{\int_{M\setminus\Omega}|\nabla u|^2e^{-f}d\sigma\leq\alpha^2\int_{M\setminus \Omega}u^2e^{-f}d\sigma}\\
	&&+(2\alpha/\delta+1/\delta^2)\left(\int_{A_{\delta}(\partial \Omega)}e^{2h_j}e^{-f}d\sigma+\int_{A_{\delta}(\partial B_r)}e^{2h_j}e^{-f}d\sigma\right).\nonumber
	\end{eqnarray} 
	For $j$ and $r$ sufficiently large such that $r>j$ and  $h_j(x)=\alpha\rho(x)$ for all  $x\in A_{\delta}(\partial \Omega)$, there exists a real constant $C$ independent of $r$ and $j$ such that
	\begin{equation}\label{ee10}
	(2\alpha/\delta+1/\delta^2)\int_{A_{\delta}(\partial \Omega)}e^{2h_j}e^{-f}d\sigma\leq C.
	\end{equation}
	
	(i) Assume that $\Vol_f(M)=+\infty$. Then
	\begin{equation}\label{ee11}
	\int_{M\setminus \Omega}u^2e^{-f}d\sigma=\int_{M\setminus \Omega}e^{2h_j}\chi_r^2e^{-f}d\sigma\rightarrow\infty \ \ \ \ {\rm as} \ \ \ r,j\rightarrow\infty.
	\end{equation}
	If $\mu_{\delta}\geq0,$ it follows from the definition of $\mu_{\delta}$ that, for any $\varepsilon>0,$ there exists a sequence $\{r_n\}$ with $r_n>2j(2+\mu_{\delta}/\varepsilon)$ for every $n$ such that
	$$
	\mu_{\delta}(r_n)=\frac{1}{r_n}\log \Vol_f(A_{\delta}(\partial B_{r_n}))\leq\mu_{\delta}+\varepsilon.
	$$
	By choosing $\alpha=\alpha(\varepsilon)=(\mu_{\delta}+2\varepsilon)/2$, we obtain
	$$h_j(x)=(2 j-\rho(x))\alpha\leq(2j-r_n)(\mu_{\delta}+2\varepsilon)/2$$
	for all $x\in A_{\delta}(\partial B_{r_n}),$
	and
	\begin{eqnarray}\label{ee13}
	\int_{A_{\delta}(\partial B_{r_n})}e^{2h_j}e^{-f}d\sigma&\leq& e^{(2j-r_n)(\mu_{\delta}+2\varepsilon)}e^{(\mu_{\delta}+\varepsilon)r_n}\nonumber\\
	&=& e^{2j(\mu_{\delta}+2\varepsilon)-r_n\varepsilon}\\
	&\leq& 1.\nonumber
	\end{eqnarray}
	Therefore, \ by (\ref{ee9}), (\ref{ee10}), (\ref{ee11}) and (\ref{ee13}), \ we can select $n$ and $j$ such that $u_n=e^{h_j}\chi_{r_n}$ and
	$$
	\frac{\displaystyle{\int_M}|\nabla u_n|^2e^{-f}d\sigma}{\displaystyle{\int_M}u_n^2e^{-f}d\sigma}\leq\alpha^2(\varepsilon)+\varepsilon_1
	$$
	for any $\varepsilon_1>0,$ where $\alpha=(\mu_{\delta}+2\varepsilon)/2$. Now, if $\mu_{\delta}<0$, then, for any $\varepsilon>0$ satisfying $\mu_{\delta+\varepsilon}<0,$ there exists a sequence $\{r_n\}$ such that $\mu_{\delta}(r_n)\leq\mu_{\delta}+\varepsilon.$ Setting $\alpha=0$ and $e^{h_j}={\bf 1},$ we have
	$$
	\int_{A_{\delta}(\partial B_{r_n})}e^{2h_j}e^{-f}d\sigma=\Vol_f(A_{\delta}(\partial B_{r_n}))\leq e^{(\mu_{\delta}+\varepsilon)r_n}<1.
	$$
	Hence, for any $\varepsilon_1,$ we can select $n$ such that 
	$$\frac{\displaystyle{\int_M}|\nabla u_n|e^{-f}d\sigma}{\displaystyle{\int_M} u_n^2e^{-f}d\sigma}\leq\varepsilon_1,$$
	where $u_n={\bf 1}\chi_{r_n}.$ 
	
	(ii) Now, assume $\Vol_f(M)<+\infty$. In this case,  $-\infty\leq\bar{\mu}_{\delta}\leq0;$ we can assume $-\infty<\bar{\mu}_{\delta}\leq0.$ It follows from the definition of $\bar{\mu}_{\delta}$ that, for any sufficiently small $\varepsilon>0,$ there exists a sequence $\{r_n\}$ such that $$r_n>2j(2\varepsilon-\bar \mu_{\delta})/(\varepsilon-2\bar \mu_{\delta})$$ and  
	$$
	\bar \mu_{\delta}-\varepsilon\leq\mu_{\delta}(r_n)=\frac{1}{r_n}\log \Vol_f(A_{\delta}(\partial B_{r_n}))\leq\bar \mu_{\delta}+\varepsilon. 
	$$ Here, we can assume that this sequence $\{r_n\}$ satisfies $r_{n+1}-r_n\geq\delta$. Choosing $\alpha=\alpha(\varepsilon)=-(\bar \mu_{\delta}-2\varepsilon)/2$ and $g(x)=e^{\alpha\rho(x)}$, we obtain
	\begin{eqnarray*}
		\int_{B_r}g^2e^{-f}d\sigma&\geq&\sum_{A(r)}\int_{B_{r_n+\delta}\setminus B_{r_n}}g^2e^{-f}d\sigma\geq\sum_{A(r)}e^{2\alpha r_n}\cdot \Vol_f(A_{\delta}(\partial B_{r_n}))\\
		&\geq&
		\sum_{A(r)}e^{2\alpha r_n}e^{(\bar\mu_{\delta}-\varepsilon)r_n}=\sum_{A(r)}e^{\varepsilon r_n}\rightarrow\infty \ \ \ \ \ {\rm as} \ \ r\rightarrow\infty,
	\end{eqnarray*}
	where $A(r)=\{n; \  r_n+\delta\leq r\};$ resulting 
	\begin{eqnarray*}
		\lefteqn{\int_{M\setminus \Omega} u^2e^{-f}d\sigma=\int_{M\setminus \Omega} e^{2h_j}\chi_r^2e^{-f}d\sigma}\\
		&\geq&\int_{B_j}g^2e^{-f}d\sigma-\int_{\Omega}g^2e^{-f}d\sigma\rightarrow\infty \ \ \ {\rm as} \ \ r,j\rightarrow\infty
	\end{eqnarray*}
	and
	$$
	\int_{A_{\delta}(\partial B_{r_n})} e^{2h_j}e^{-f}d\sigma\leq e^{(2j-r_n)(2\varepsilon-\bar \mu_{\delta})}e^{(\bar \mu_{\delta}+\varepsilon)r_n}\leq1.
	$$
	In the same way as in the case of infinite weighted volume, selecting sufficiently large $n$ and $j$, we obtain the desired estimate. \qed
\end{proof}

Observe that $\mu_{\delta_1}\leq\mu_{\delta_2}$ and $\bar\mu_{\delta_1}\leq\bar\mu_{\delta_2}$ if $\delta_1<\delta_2$. Thus, there exists the limits
$$\mu_0=\lim_{\delta\rightarrow0}\mu_{\delta} \ \ \ \ \ \ {\rm and} \ \ \ \ \ \ \ 
\bar\mu_0=\lim_{\delta\rightarrow0}\bar\mu_{\delta},$$
with $\mu_{\delta}$ defined in (\ref{a13}).

\begin{lemma}\label{al2} Let $M_f$ be a complete noncompact weighted manifold. Then $$\inf\sigma_{ess}(-\Delta_f)\leq\mu^2/4,$$ where $\mu=\max\{\mu_0,0\}$ if $\Vol_f(M)=+\infty$ and $\mu=\bar{\mu}_0$ if $\Vol_f(M)<+\infty$.	
\end{lemma}

\begin{proof}
	It follows direct from Lemma \ref{al1} that $\inf\sigma_{ess}(-\Delta_f)\leq\mu_0^2/4$	if $M_f$ has infinite weighted volume and $\inf\sigma_{ess}(-\Delta_f)\leq\bar\mu_0^2/4$ if $M_f$ has finite weighted volume. Now, if $\Vol_f(M)=+\infty$ and $\mu_0=\lim_{\delta\rightarrow0}\mu_{\delta}<0,$ then there exists $\delta>0$ such that $\mu_{\delta}<0.$ Using Lemma \ref{al1} (i), we obtain $\inf\sigma_{ess}(-\Delta_f)=0.$
	This concludes the proof of this lemma.\qed	
\end{proof}	

We are going to use Lemma \ref{al1} and Lemma \ref{al2} to prove Theorem \ref{it4} that was enunciate in Introduction.	

\begin{proof}{\it of Theorem \ref{it4}} \  (i) Assume $\Vol_f(M)=+\infty$. Recall that
	$$\mu_{\delta}=\liminf_{r\rightarrow\infty}\frac{1}{r}\log \Vol_f(A_{\delta}(\partial B_r)) \ \ {\rm and} \ \ \ \mu_v=\liminf_{r\rightarrow\infty}\frac{1}{r}\log\Vol_f(B_r).$$ Since $\mu_{\delta_1}\leq\mu_{\delta_2}$ whenever $\delta_1<\delta_2,$ then $\mu_0=\lim_{\delta\rightarrow0}\mu_{\delta}(r)$ satisfies $\mu_0\leq\mu_{\delta}$. 
	For arbitrary fixed $\delta>0$, we can select $r$ sufficiently large such that
	$$\Vol_f(A_{\delta}(\partial B_r))\leq\Vol_f(B_r).$$
	Thus, $\mu_{\delta}\leq\mu_v$ and $\mu_0\leq\mu_{\delta}\leq\mu_v$ for all $\delta>0.$ If $\mu_0\geq0,$ then it follows from Lemma \ref{al2} that
	$$\inf\sigma_{ess}(-\Delta_f)\leq\frac{\mu_0^2}{4}\leq\frac{\mu_v^2}{4}.$$
	Moreover,  $\inf\sigma_{ess}(-\Delta_f)=0$ if $\mu_0<0$.
	
	(ii) Assume that $\Vol_f(M)<+\infty$. Recall that
	$$\bar\mu_{\delta}=\limsup_{r\rightarrow\infty}\frac{1}{r}\log \Vol_f(A_{\delta}(\partial B_r))$$
	and $$\mu_w=\liminf_{r\rightarrow\infty}\frac{-1}{r}\log(\Vol_f(M)-\Vol_f(B_r)).$$
	Since $\bar\mu_{\delta_1}\leq\bar\mu_{\delta_2}$ if $\delta_1<\delta_2,$ then  $\bar\mu_0=\lim_{\delta\rightarrow0}\bar\mu_{\delta},$ 
	satisfies 
	\begin{equation}\label{eee3}
	\bar\mu_0\leq\bar\mu_{\delta}\leq\lim_{r\rightarrow\infty}\sup\frac{1}{r}\log(\Vol_f(M)-\Vol_f(B_r))=-\mu_w\leq0.	
	\end{equation}
	The second inequality of the above expression follows from the inequality
	$$\Vol_f(A_{\delta}(\partial B_r))\leq\Vol_f(M)-\Vol_f(B_r).$$
	Let $\bar{\mu}_{\delta}<0.$ For any $\varepsilon>0$ satisfying $\bar \mu_{\delta}+\varepsilon<0,$ there exists $r_0$ such that  $$\frac{1}{r}\log \Vol_f(A_{\delta}(\partial B_r))<\bar \mu_{\delta}+\varepsilon$$
	for any $r\geq r_0.$ Then we get, for any $r\geq r_0,$
	$$
	\Vol_f(M)-\Vol_f(B_r)=\sum_{k=0}^{\infty}\Vol_f(A_{\delta}(\partial B_{r+k\delta}))\leq\sum_{k=0}^{\infty}e^{(\bar{\mu}_{\delta}+\varepsilon)(r+k\delta)}.
	$$
	Thus, we have 
	$$\frac{1}{r}\log(\Vol_f(M)-\Vol_f(B_r))\leq\bar{\mu}_{\delta}+\varepsilon-\frac{1}{r}\log(1-e^{(\bar{\mu}_{\delta}+\varepsilon)\delta}),$$
	resulting in $-\mu_w\leq\bar{\mu}_{\delta}+\varepsilon.$  Since we can select arbitrary small $\varepsilon>0,$ then $-\mu_w\leq\bar{\mu}_{\delta}.$ Therefore, by (\ref{eee3}), $-\mu_w=\bar \mu_{\delta}=\bar\mu_{0}.$ Besides on, $\bar \mu_{\delta}=-\mu_w=0$ if $\bar{\mu}_{\delta}=0$. Therefore, $|\bar{\mu}_{0}|={\mu}_{w}.$ This concludes the part (ii).
	
	Let  $M_f=(\R^{n},ds^2,e^{-f}d\sigma)$ with $$f=\frac{r}{2} \ \ \ \ {\rm and} \ \ \ \ ds^2=dr^2+g^2(r)d\theta^2$$
	such that $g:[0,+\infty)\rightarrow\R$ is a nonnegative smooth function satisfying
	$$
	g(0)=0, \ \ \ g'(0)=1, \ \ \ {\rm and} \ \ \ g(r)=e^{-\frac{r}{2(n-1)}} \ \ {\rm for \ all} \ r\geq r_0>0.
	$$
	Note that 
	\begin{eqnarray*}
		\Delta_fu&=&\Delta\left( e^{\frac{1 }{2}r}\right)-\frac{1}{2}\left\langle\nabla r,\nabla\left(e^{\frac{1}{2}r}\right)\right\rangle\\
		&=&\frac{1}{2}e^{\frac{1}{2}r}\Delta r+\frac{1}{4}e^{\frac{1}{2}r}|\nabla r|^2-\frac{1}{4}e^{\frac{1}{2}r}|\nabla r|^2\\
		&=&\frac{1}{2}e^{\frac{1}{2}r}\Delta r,
	\end{eqnarray*}
	where 
	$$\Delta r=({n-1})\frac{\frac{d}{dr}\left(e^{-\frac{r}{2(n-1)} }\right)}{e^{-\frac{r}{2(n-1)} }}=-\frac{1}{2}$$
	for all $r\geq r_0>0$. 
	Therefore, $u$ is a positive solution of
	$$\Delta_fu+\frac{1}{4}u=0$$
	for all $r\geq r_0>0.$ Since, see Corollary 6.4 of \cite{BessaPigolaSetti2013},
	$$\lambda_1^f(M\setminus B_r)\geq\inf_{M\setminus B_r}\left\{-\frac{\Delta_f u}{u}\right\},$$
	for any positive function, then
	\begin{equation}\label{e111}\lambda_1^f(M\setminus B_r)\geq\frac{1}{4},
	\end{equation}	
	for all $r\geq r_0>0.$ Using Remark \ref{ri1}, we have that 
	\begin{eqnarray*}
		\mu_w&=&\liminf_{r\rightarrow\infty}\frac{-1}{r}\log(\Vol_f(M)-\Vol_f(B_r))\\
		&=&\liminf_{r\rightarrow\infty}\frac{-1}{r}\log\left(\omega_ne^{-r}\right)\\
		&=&1.
	\end{eqnarray*}
	Thus, it follows from (\ref{e111}) and (\ref{1111}), and the first part of Theorem \ref{it4} that $$\frac{1}{4}\leq\lambda_1^f(M\setminus B_r)\leq\inf \sigma_{ess}(-\Delta_f)\leq\frac{1}{4}=\frac{\mu_w^2}{4}.$$
	This concludes the proof of Theorem \ref{it4}.	\qed	
\end{proof}

We are interested in estimating $\lambda^f_1(M\setminus\Omega)$, where $\Omega$ is an arbitrary compact subset of $M$. For this, we will look at the oscillatory behaviour of solutions of one second order ordinary differential equation. We will need the following result obtained by Fite (see \cite{Fite1918}, Theorem I):

\begin{theorem}[\cite{Fite1918}] \label{t5} Assume that $q$ and $\lambda$ are continuous functions of $t\in[t_0,+\infty)$ such that $|q(t)|\leq \alpha$ and $\lambda(t)\geq h>0$ for every $t\geq t_0$, where $\alpha$ and $h$ are constants such that $4h-\alpha^2>0$, then every solution of differential equation
	\begin{equation}\label{e5}
	y''+qy'+\lambda y=0
	\end{equation}
	changes sign an infinite number of times, i.e., every solution of {\normalfont (\ref{e5})} is oscillatory. 
\end{theorem}

\begin{remark}
	The oscillation criteria in the integral form was firstly used by do Carmo and Zhou \cite{doCarmoZhou1999} to obtain estimates of $\lambda_1(M\setminus\Omega)$ of the Laplacian operator supposing polynomial or exponential growth of volume without weight. Other authors have also used this criterion to estimating the first eigenvalue of other elliptical operators assuming volume conditions, for example, see \cite{BianchiniMariRigoli2009}, \cite{Elbert2002}, and \cite{ImperaRimoldi2015}. In this last reference, Impera and Rimoldi have given an upper estimates of $\lambda_1^{f}(M\setminus\Omega)$ of the weighted Laplacian operator supposing polynomial or exponential growth of weighted volume and infinite weighted volume. 
\end{remark}

Now, we can use the previously oscillation result to prove Theorem \ref{it1}.

\begin{proof}{\it of Theorem \ref{it1}}
	Let $v(t)=\vol_f(\partial B_{t})$ denote the weighted area of geodesic sphere $\partial B_t$ with center in $o\in M$ and radius $t>0$. Then, by coarea formula, we have
	$$\Vol_f(B_r)=\int_{B_r}e^{-f}d\sigma=\int_0^r\int_{\partial B_t}e^{-f}dA \,dt=\int_{0}^{r}v(t)dt.$$
	Choose $t_0\geq r_0$ such that $\Omega\subset B_{t_0}$. Put $q(r):=\displaystyle{\frac{v'(r)}{v(r)}}, \ r\geq t_0$. Since by hypothesis 
	$$|q(r)|=\left|\frac{d}{dr}\left(\log v(r)\right)\right|\leq \alpha$$
	for all $r\geq t_0$, we obtain that the differential equation
	\begin{equation}\label{ee7}
	y''(r)+\frac{v'(r)}{v(r)}y'(r)+\lambda y(r)=0, \ \ \ \ r\geq t_0,
	\end{equation}
	satisfies the condition of Theorem \ref{t5}. Therefore, for any $\lambda>0$ such that $$4\lambda-\alpha^2>0,$$ the differential equation (\ref{ee7}) is oscillatory, i.e., every solution $y(r)$, $r\in[t_0,+\infty)$, of (\ref{ee7}) changes sign an infinite number of times. Fixed the initial values $y(t_0)=y_0$ and $y'(t_0)=y_0'$, we have that any solution of (\ref{ee7}) can be extended to $[t_0,+\infty)$. Let $y:[t_0,+\infty)\rightarrow\R$ a non-trivial oscillatory solution of (\ref{ee7}) with $v(r)=\vol_f(\partial B_r)$. Thus, there exists two numbers $r_1$ and $r_2$ in $[t_0,+\infty)$ such that $r_1<r_2$ and $y(r_1)=y(r_2)=0$, and $y(t)\neq0$ for any $t\in(r_1,r_2)$. Write $r(x)=dist(x,o)$, $\varphi(x)=y(r(x))$ and $\Omega_{\lambda}=B_{r_2}\setminus B_{r_1}$. It follows by coarea formula that
	\begin{eqnarray*}
		0\leq\lambda_1^f(M\setminus\Omega)&\leq&\lambda_1^f(\Omega_{\lambda})\leq\frac{\displaystyle{\int_{\Omega_{\lambda}}}|\nabla\varphi|^2e^{-f}d\sigma}{\displaystyle{\int_{\Omega_{\lambda}}}|\varphi|^2e^{-f}d\sigma}
		=\frac{\displaystyle{\int_{r_1}^{r_2}}(y'(t))^2v(t)\,dt}{\displaystyle{\int_{r_1}^{r_2}}(y(t))^2v(t)\,dt}.\\
	\end{eqnarray*}
	However, by equation (\ref{ee7}),
	\begin{eqnarray*}(yvy')'=(y')^2v+yv'y'+yvy''=(y')^2v+\left(y''+\frac{v'}{v}y'\right)yv=(y')^2v-\lambda y^2v,
	\end{eqnarray*}
	resulting
	$$0\leq\lambda_1^f(M\setminus\Omega)\leq\frac{\displaystyle{\int_{r_1}^{r_2}}\lambda(y(t))^2v(t)\,dt}{\displaystyle{\int_{r_1}^{r_2}}(y(t))^2v(t)\,dt}=\lambda.$$ Since $\lambda$ is an arbitrary positive constant larger than $\displaystyle{\frac{\alpha^2}{4}}$, then $$\lambda_1^f(M\setminus\Omega)\leq\displaystyle{\frac{\alpha^2}{4}},$$
	by implying, see (\ref{1111}), that
	$$\inf\sigma_{ess}(-\Delta_f)\leq\displaystyle{\frac{\alpha^2}{4}}.$$
	
	Now, let $M_f=(\R^{n},ds^2,e^{-f}d\sigma$) with $$f=\frac{\alpha r}{2} \ \ \ \ {\rm and} \ \ \ \ ds^2=dr^2+g^2(r)d\theta^2$$
	such that $g:[0,+\infty)\rightarrow\R$ is a nonnegative smooth function satisfying
	$$
	g(0)=0, \ \ \ g'(0)=1, \ \ \ {\rm and} \ \ \ g(r)=e^{-\frac{\alpha}{2(n-1)} r} \ \ {\rm for \ all} \ r\geq r_0>0.
	$$ 
	Now, we consider the function $u=e^{\frac{\alpha}{2}r}$. Note that 
	\begin{eqnarray*}
		\Delta_fu&=&\Delta\left( e^{\frac{\alpha }{2}r}\right)-\frac{\alpha}{2}\left\langle\nabla r,\nabla\left(e^{\frac{\alpha}{2}r}\right)\right\rangle\\
		&=&\frac{\alpha}{2}e^{\frac{\alpha}{2}r}\Delta r+\frac{\alpha^2}{4}e^{\frac{\alpha}{2}r}|\nabla r|^2-\frac{\alpha^2}{4}e^{\frac{\alpha}{2}r}|\nabla r|^2\\
		&=&\frac{\alpha}{2}e^{\frac{\alpha}{2}r}\Delta r,
	\end{eqnarray*}
	where 
	$$\Delta r=({n-1})\frac{\frac{d}{dr}\left(e^{-\frac{\alpha}{2(n-1)} r}\right)}{e^{-\frac{\alpha}{2(n-1)} r}}=-\frac{\alpha}{2}$$
	for all $r\geq r_0>0$. Therefore, $u$ is a positive solution of
	$$\Delta_fu+\frac{\alpha^2}{4}u=0$$
	for all $r\geq r_0>0.$
	Since, see Corollary 6.4 of \cite{BessaPigolaSetti2013},
	$$\lambda_1^f(M\setminus B_r)\geq\inf_{M\setminus B_r}\left\{-\frac{\Delta_f u}{u}\right\},$$
	for any positive function, then
	$$\lambda_1^f(M\setminus B_r)\geq\frac{\alpha^2}{4}.
	$$	
	Therefore,
	\begin{equation}\label{e11}\lambda_1^f(M\setminus\Omega)\geq\lambda_1^f(M\setminus B_{r_0})\geq\frac{\alpha^2}{4}
	\end{equation}
	for any compact $\Omega\supset B_{r_0}.$ Using Remark \ref{ri1}, we have that 
	$$\left|\frac{d}{dr}\left(\log\vol_f(\partial B_r)\right)\right|=|-\alpha|=\alpha \ \ \ {\rm for \ all} \ r\geq r_0>0.$$
	Thus, it follows from Theorem \ref{it1} that $$\lambda_1(M\setminus \Omega)\leq \frac{\alpha ^2}{4}.$$ Using inequalities (\ref{e11}) along with the above inequality and equality (\ref{1111}), we conclude that $$\lambda_1(M\setminus \Omega)=\inf\sigma_{ess}(-\Delta_f)=\frac{\alpha^2}{4}$$
	for any compact $\Omega\supset B_{r_0}.$ \qed
\end{proof}

\begin{example}\label{ri3} There is a class of complete Riemannian manifolds that satisfies the conditions of Theorem \ref{it1}. In fact, let $M_f^{n}=(\R^{n},ds^2,e^{-f}d\sigma)$ with weight $f=f(r)$ and a smooth metric $ds^2=dr^2+g^2(r)d\theta^2$, where $d\theta^2$ denotes the standard metric on $(n-1)$-dimensional unit sphere $\s^{n-1}_1$ and $g:[0,+\infty)\rightarrow\R$ is a nonnegative smooth function satisfying $g(0)=0$, $g'(0)=1$, and
	$$\left|(n-1)\frac{g'(r)}{g(r)}-f'(r)\right|\leq\alpha, \ \ \ \ \alpha>0$$
	for all $r\geq r_0>0.$ Thus, 
	\begin{equation}\label{er1}
	\Vol_f(B_r)=\int_0^r\int_{\s^{n-1}_1}(g(t))^{n-1}e^{-f(t)}d\theta dt.
	\end{equation}
	and
	$$\vol_f(\partial B_r)=\omega_{n}(g(r))^{n-1}e^{-f(r)},$$ 
	where $\omega_n$ is $({n-1})$-dimensional volume of $\s^{n-1}_1$. Therefore,  
	\begin{eqnarray}\label{er2}
	\frac{d}{dr}\log(\vol_f(\partial B_r))&=&\frac{d}{dr}\log\left(\omega_n(g(r))^{n-1}e^{-f(r)}\right)\nonumber\\
	&=&(n-1)\frac{g'(r)}{g(r)}-f'(r)
	\end{eqnarray}
	for all $r\geq r_0.$	
\end{example}

\section{Estimates for the Weighted Mean Curvature}

Let ${\overline M}_f^{n+1}$ be a weighted manifold, i.e., $${\overline M}_f^{n}=( \overline M^{n},\langle\,,\,\rangle,e^{-f}d\mu).$$ Let $x:M^n\rightarrow\overline M^{n+1}_f$ be an isometric immersion of a Riemannian orientable manifold $M^n$ into weighted manifold $\overline M^{n+1}_f.$ The function $f:\overline M \rightarrow\R$  restricted to $M$ induces a weighted measure $e^{-f}d\sigma$ on $M$. Thus, we have an induced weighted manifold $M_f^n=(M,\langle\,,\,\rangle,e^{-f}d\sigma)$.

The {\it second fundamental form} $A$ of $x$ is defined by
$$A(X,Y)=(\overline\nabla_XY)^{\perp}, \ \ \ \ \ \ X,Y\in T_pM, \ \ p\in M,$$
where $\perp$ denotes the projection above the normal bundle of $M$. 

The {\it weighted mean curvature vector} of $M$ is defined by
$${\bf H}_f={\bf H}+(\overline\nabla f)^{\perp}$$
with ${\bf H}=\tr A.$ The hypersurface $M$ is called {\it $f$-minimal} when its weighted mean curvature vector ${\bf H}_f$ vanishes identically; when there exists real constant $C$ such that ${\bf H}_f=-C\eta$ with $\eta$ being unit normal vector field, we say the hypersurface $M$ has {\it constant weighted mean curvature}. 

Let $F:(-\varepsilon,\varepsilon)\times M\rightarrow \overline M_f$, $F_f(p)=F(t,p)$ for all $t\in(-\varepsilon,\varepsilon)$ and $p\in M,$ be a variation of the immersion $x$ associated with the normal vector field $u\eta$, where $u\in C^{\infty}_c(M).$  The corresponding variation of the {\it functional weighted area} $\mathcal A_f(t)=\Vol_f(F_t(M))$ satisfies 
\begin{equation} \label{ie1}
\mathcal A_f'(0)=\int_M H_f ue^{-f}d\sigma,	
\end{equation}
where $H_f$ is such that ${\bf H}_f=-H_f\eta$. The expression (\ref{ie1}) is known as {\it first variation formula}. 

\begin{remark}
	In \cite{McGR2015}, McGonagle and Ross obtained the first variation formula to hypersurface $M^{n}\subset \R^{n+1}_f,$ i.e, in the case where $\overline M_f=\R^{n+1}_f$ (see \cite{McGR2015}, (2.1), pp. 282). To the general case, we can proceed in the same way as \cite{McGR2015}.  	
\end{remark}	

The $f$-minimal hypersurfaces are critical points of the functional weighted area. Yet, the hypersurfaces with constant weighted mean curvature can be viewed as critical points of the functional weighted area restricted to variations which preserve the {\it enclose weighted volume,} i.e., to functions $u\in C^{\infty}_c(M)$ which satisfy the additional condition $$\int_Mue^{-f}d\sigma=0.$$ For such critical points, the {\it second variation} of the functional weighted area is given by
$$\mathcal A_f''(0)=-\int_M\left(u\Delta_fu+\left(|A|^2+\overline{\Ric_f}(\eta,\eta)\right)u^2\right)d\sigma,$$
where $\overline{\Ric}_f$ is the Bakry-\' Emery Ricci curvature and $A$ is the second fundamental form. 

\begin{remark}
	The second variation formula of functional weighted area, $\mathcal A_f''(0)$, can be viewed in \cite{McGR2015}, Lemma 2.3, to isometric immersion of hypersurface $M^n$ into $\R^{n+1}_f$ with constant weighted mean curvature. The case which $\overline M^{n+1}_f$ is any weighted manifold, it follows from analogue arguments shown in the proof of Lemma 2.3, \cite{McGR2015}. When $f$ is a constant function, the first and second variation formula were given by Barbosa and do Carmo \cite{BarbosadoCarmo1984} and Barbosa, do Carmo and Eschenburg \cite{BarbosadoCarmoEschenburg1988}.    
\end{remark}

The operator 
$$L_f=\Delta_f+|A|^2+\overline{\Ric_f}(\eta,\eta)$$
is called the {\it $f$-stability operator} of the immersion $x$. In the $f$-minimal case, the $f$-stability operator is viewed as acting on $\mathcal F=C_c^{\infty}(M)$; in the case of the hypersurfaces with constant weighted mean curvature, the $f$-stability operator is viewed as acting on $$\mathcal F=C_c^{\infty}(M)\cap \left\{u\in C^{\infty}_c(M); \ \int_Mue^{-f}d\sigma=0\right\}.$$
Associated with $L_f$ is the quadratic form
$$I_f(u,u)=-\int_MuL_fue^{-f}d\sigma.$$
For each compact domain $\Omega\subset M$, define the index, $\Ind_f\Omega$, of $L_f$ in $\Omega$ as the maximal dimension of a subspace of $\mathcal F$ where $I_f$ is a negative definite. The {\it index}, $\Ind_fM$, of $L_f$ in $M$ (or simply, the index of $M$) is then defined by
$$\Ind_fM=\sup_{\Omega\subset M}\Ind_f\Omega,$$
where the supreme is taken over all compact domains $\Omega\subset M.$  For more details, see \cite{ChengZhou2015}.

\begin{remark} In the case of constant weighted mean curvature immersion, if we analyze the index of $L_f$ on $M$, with $L_f$ acting on $C_c^{\infty}(M)$ instant of $C_c^{\infty}(M)\cap \left\{u\in C_c^{\infty}(M); \ \int_Mue^{-f}d\sigma=0\right\}$, we denote the {\it stronger index} of $L_f$ on $M$ by $\ind_fM$. \ However, it is easy to verify that $Ind_f(M)<\infty$ is equivalent to $ind_{f}(M)<\infty,$ so in the statement of Theorem \ref{it5} and Theorem \ref{it3} along with their respective corollaries, they are immaterial whether one takes $\Ind_fM<\infty$ or $ind_fM<\infty.$
\end{remark}

We will demonstrate Theorem \ref{it5} and Theorem \ref{it3}. To this, it is important to know the following inequality.  

\begin{lemma}\label{l1} Let $m$ be a real number such that $m<-1$ or $m>0$. Then 
	$$(a+b)^2\geq\frac{a^2}{1+m}-\frac{b^2}{m}$$
	for all $a,b\in \R$.	
\end{lemma}

\begin{proof}
	Initially, note that
	$$\left(\frac{1}{k}a+kb\right)^2\geq0, \ \ \ \ \ {\rm where} \ \ k=\sqrt{\frac{1+m}{m}}.$$
	Therefore,
	$$0\leq \frac{1}{k^2}a^2+2ab+k^2b^2=\left(\frac{1}{k^2}-1\right)a^2+(k^2-1)b^2+(a+b)^2,$$
	this is
	$$(a+b)^2\geq\left(1-\frac{1}{k^2}\right)a^2+(1-k^2)b^2=\frac{a^2}{1+m}-\frac{b^2}{m}.$$ \qed
\end{proof}

\begin{proof}{\it of Theorem \ref{it5}} 
	Since  $\ind_fM<\infty$ by hypothesis, then using Proposition 5 of \cite{ImperaRimoldi2015}, there exists a compact set $\Omega$ and a positive function $u$ on $M$ such that 
	$$0=L_fu=\Delta_fu+|A|^2u+\overline{\Ric}_f(\eta,\eta)u$$
	on $M\setminus \Omega$. Let $o\in M$ and let $r_0>0$ be such that $\Omega\subset B_{r_0}(o).$ Therefore, it follows from equality (\ref{1111}) and Theorem \ref{it4}, $$\lambda_1^f(M\setminus B_{r})\leq\inf\sigma_{ess}(-\Delta_f)\leq\frac{\mu^2}{4},$$ 
	where $\mu=\mu_v$ if $\Vol_f(M)=+\infty$ and $\mu=\mu_w$ if $\Vol_f(M)<+\infty$. Now, by using Corollary 6.4 of \cite{BessaPigolaSetti2013} and the previously inequality, we obtain
	\begin{eqnarray*}
		\frac{\mu^2}{4}&\geq&\lambda_1^f(M\setminus B_{r})\geq\inf_{M\setminus B_{r}}\left(-\frac{\Delta_fu}{u}\right)\geq\inf_{M\setminus B_{r}}\left\{\frac{H^2}{n}+\overline{\Ric}_f(\eta,\eta)\right\},
	\end{eqnarray*} 
	because $|A|^2\geq\displaystyle{\frac{H^2}{n}}$.
	Let $m>0$ and since $H_f=H-\langle\overline\nabla f,\eta\rangle=C,$ then it follows from Lemma \ref{l1} that
	$$H^2=((H_f)+\langle\overline\nabla f,\eta\rangle)^2\geq\frac{H_f^2}{1+m}-\frac{\langle\overline\nabla f,\eta\rangle^2}{m}$$
	and
	$$H^2=((H_f)+\langle\overline\nabla f,\eta\rangle)^2\geq\frac{\langle\overline\nabla f,\eta\rangle^2}{1+m}-\frac{H_f^2}{m};$$ resulting in	
	\begin{eqnarray*}
		\frac{\mu^2}{4}&\geq&\lambda_1^f(M\setminus B_{r})\geq\inf_{M\setminus B_{r}}\left\{\frac{H_f^2}{n(1+m)}-\frac{\langle\overline\nabla f,\eta\rangle^2}{nm}  +\overline{\Ric}_f(\eta,\eta)\right\}
	\end{eqnarray*} 
	and
	\begin{eqnarray*}
		\frac{\mu^2}{4}&\geq&\lambda_1^f(M\setminus B_{r})\geq\inf_{M\setminus B_{r}}\left\{\frac{\langle\overline\nabla f,\eta\rangle^2}{n(1+m)}-\frac{H_f^2}{nm}+\overline{\Ric}_f(\eta,\eta)\right\}.
	\end{eqnarray*} 
	Therefore,
	$$\frac{H^2_f}{nm}\geq-\frac{\mu^2}{4}+\inf_{M\setminus B_{r}}\left\{\overline{\Ric}_f(\eta,\eta)+\frac{\langle\overline\nabla f,\eta\rangle^2}{n(1+m)}\right\}$$
	and
	$$\frac{H^2_f}{n(1+m)}\leq\frac{\mu^2}{4}-\inf_{M\setminus B_{r}}\left\{\overline{\Ric}_f^{nm}(\eta,\eta)\right\}.$$ 
	The second part of this theorem is immediate, just use the previous inequality by doing $$\overline{\Ric}_f^{nm}(\eta,\eta)\geq\frac{\mu^2}{4}.$$ \qed
\end{proof}

\begin{proof}{\it of Theorem \ref{it3}}  Since $\ind_fM<\infty$ by hypothesis, then using Proposition 5 of \cite{ImperaRimoldi2015}, there exists a compact set $\Omega$ and a positive function $u$ on $M$ such that 
	$$0=L_fu=\Delta_fu+|A|^2u+\overline{\Ric}_f(\eta,\eta)u$$
	on $M\setminus \Omega$. Let $o\in M$ and let $r_0>0$ be such that $\Omega\subset B_{r_0}(o).$ Therefore, it follows from Theorem \ref{it1} that $$\lambda_1^f(M\setminus B_{r})\leq\inf\sigma_{ess}(-\Delta_f)\leq\frac{\alpha^2}{4}.$$ Now, by using Corollary 6.4 of \cite{BessaPigolaSetti2013} and previously inequality, we obtain
	\begin{eqnarray*}
		\frac{\alpha^2}{4}&\geq&\lambda_1^f(M\setminus B_{r})\geq\inf_{M\setminus B_{r}}\left(-\frac{\Delta_fu}{u}\right)\geq\inf_{M\setminus B_{r}}\left\{\frac{H^2}{n}+\overline{\Ric}_f(\eta,\eta)\right\},
	\end{eqnarray*} 
	because $|A|^2\geq\displaystyle{\frac{H^2}{n}}$.
	Let $m>0$ and since $H_f=H-\langle\overline\nabla f,\eta\rangle=C,$ then it follows from Lemma \ref{l1} that
	$$H^2=((H_f)+\langle\overline\nabla f,\eta\rangle)^2\geq\frac{H_f^2}{1+m}-\frac{\langle\overline\nabla f,\eta\rangle^2}{m}$$
	and
	$$H^2=((H_f)+\langle\overline\nabla f,\eta\rangle)^2\geq\frac{\langle\overline\nabla f,\eta\rangle^2}{1+m}-\frac{H_f^2}{m};$$
	resulting in	
	\begin{eqnarray*}
		\frac{\alpha^2}{4}&\geq&\lambda_1^f(M\setminus B_{r})\geq\inf_{M\setminus B_{r}}\left\{\frac{H_f^2}{n(1+m)}-\frac{\langle\overline\nabla f,\eta\rangle^2}{nm}  +\overline{\Ric}_f(\eta,\eta)\right\}
	\end{eqnarray*} 
	and
	\begin{eqnarray*}
		\frac{\alpha^2}{4}&\geq&\lambda_1^f(M\setminus B_{r})\geq\inf_{M\setminus B_{r}}\left\{\frac{\langle\overline\nabla f,\eta\rangle^2}{n(1+m)}-\frac{H_f^2}{nm}+\overline{\Ric}_f(\eta,\eta)\right\}.
	\end{eqnarray*} 
	Therefore,
	$$\frac{H^2_f}{nm}\geq-\frac{\alpha^2}{4}+\inf_{M\setminus B_{r}}\left\{\overline{\Ric}_f(\eta,\eta)+\frac{\langle\overline\nabla f,\eta\rangle^2}{n(1+m)}\right\}$$
	and
	$$\frac{H^2_f}{n(1+m)}\leq\frac{\alpha^2}{4}-\inf_{M\setminus B_{r}}\left\{\overline{\Ric}_f^{nm}(\eta,\eta)\right\}.$$ 
	The second part of this theorem is immediate, just use the previous inequality by doing $$\overline{\Ric}_f^{nm}(\eta,\eta)\geq\frac{\alpha^2}{4}.$$ \qed	
\end{proof}

%
%

\end{document}